\documentclass{book}
\usepackage[
 	lang = american,
 	mode = online, 
]{ecr}

\usepackage{chngcntr}
\counterwithout{figure}{section}

\usepackage{labelfig}
\usepackage{graphicx}
 
\theoremstyle{plain}
 \newtheorem{theorem}{Theorem}[section]
 \newtheorem{lemma}[theorem]{Lemma}
  \newtheorem{proposition}{Proposition}[section]
    \newtheorem{corollary}{Corollary}[section]

\theoremstyle{definition}

 \newcommand{\M}{\mathcal{M}}
 \newcommand{\T}{\mathcal{T}}
 
 \newcommand{\Mod}{\mathrm{Mod}}
 \newcommand{\sys}{\mathrm{sys}}
 \newcommand{\arcsinh}{\mathrm{arcsinh}}

\begin{document}

\mainmatter


\chapter[The topological types of length bounded multicurves]{The topological types of length bounded multicurves}


\author{Hugo Parlier}
\authormark{H. Parlier}


\keywords{hyperbolic surfaces, curves, Teichm\"uller spaces, moduli spaces}
\subjclass{Primary 32G15, 52K20; secondary 53C22, 30F60}




\begin{abstract}
This article discusses inequalities on lengths of curves on hyperbolic surfaces. In particular, a characterization is given of which topological types of curves and multicurves always have a representative that satisfies a length inequality that holds over all of moduli space.
\end{abstract}


\section{Introduction}

Length inequalities for curves play an important role in the understanding of hyperbolic surfaces and their moduli spaces. A prime example of this is a theorem of Bers, which states that any closed hyperbolic surface admits a pants decomposition of length bounded by a constant which only depends on the topology of the surface, and not its geometry. Short pants decompositions have been very useful in the understanding of the underlying Teichm\"uller space, and its large to medium scale geometry. This result generalizes bounds on the length of the shortest non-trivial closed curve (the systole). These are both examples of families of curves or multicurves that admit upper bounds over Teichm\"uller or moduli space of a given topological surface. More specifically, these results tell you that any hyperbolic surface of given topology admits a curve, or a multicurve, taken among a family of topological types, and which satisfies a certain length bound.

The goal of this short note is to characterize which families of curves or multicurves admit such upper bounds on their lengths. By multicurve, we mean a finite union of curves, all of them considered up to free homotopy, and by length we mean the length of a minimum length, thus geodesic, representative. The bounds are only allowed to depend on the topology of the underlying surface. Because of the nature of the game, we are only interested in the topological type of the curves, or said differently, in its mapping class group orbit. In order to try and satisfy the length inequality, we are allowed to choose curves of minimal length in the orbit. 

The result is stated in terms of length functions. To a multicurve, we can associate the function, which takes a hyperbolic surface in the Teichm\"uller $\T(\Sigma)$ of a finite-type surface orientable $\Sigma$, and which associates to it the length of the (unique) geodesic in the free homotopy class of the multicurve. The existence of the upper bounds described depends on whether certain length functions are bounded over Teichm\"uller or moduli space.

\begin{theorem}\label{thm:main}
Let $\Gamma$ be a multicurve on $\Sigma$. Then the quantity
$$\max_{X\in \T(\Sigma)} \min_{\phi \in \Mod(\Sigma)} \ell_{\phi(\Gamma)}(X)$$
is finite if and only if, for every pants decomposition $P$, there exists $\phi$ in the mapping class group  $\Mod(\Sigma)$ such that $i(P, \phi(\Gamma))) = 0$.
\end{theorem}
For simplicity, it is stated only for multicurves but more generally it holds for families of multicurves (see Theorem \ref{thm:maingeneral}). In more colloquial terms, what the result says is that if you are given a type of multicurve, then the function that associates to any hyperbolic surface the multicurve of minimal length of that type, is a bounded function over moduli space if and only if there is a multicurve of that type disjoint from any pants decomposition. 

Note that, given a lower bound on the systole, by compactness of pinched moduli space, there is similar conditional length inequality statement (Proposition \ref{prop:conditionallength}) that holds for any multicurve, but the implied constants depend on topology, curve type and the lower bound of the systole.

Note that here we are only interested in the very first value in a subset of the length spectrum. This is obviously very different from recent results about the asymptotic growth of curves of a given type (see \cite{Erlandsson-Souto} and references therein). This is also very different from the related problems of finding precise constants, and in particular, exploring surfaces that are extremal for different geometric quantities. The focus here is on understanding what type of inequalities are possible. This is of course inspired by the many uses that have been made of these, or related, inequalities in the study of the large or medium scale of Teichm\"uller spaces with different metrics \cite{Brock,Cavendish-Parlier} for the Weil-Petersson metric, \cite{Rafi,Rafi-Tao,Papadopoulos-Theret} for the Teichm\"uller and Thurston metrics. 

Finally, note that this note is about upper bounds of length functions. If one replaces the $\max$ with a $\min$ in the above theorem, the quantity is strictly positive if and only if the multicurve has intersection (coming from a non-simple closed curve or pairwise intersecting curves). This is a consequence of the collar lemma \cite{Keen} and generalisations \cite{BasmajianCollar}. Lower bounds that depend on curve type have been studied in some detail by Basmajian \cite{BasmajianUniversal}.\\

\noindent{\bf Organization.} The article is organized as follows. The second section is mainly notation and definitions, and includes Proposition \ref{prop:conditionallength}. The third and final section is dedicated to Theorem \ref{thm:maingeneral} and ends with a discussion of which previously known length inequalities fall within its framework.

\section{Preliminaries and setup}
Throughout $\Sigma$ will be an orientable finite-type surface with $\chi(\Sigma) <0$. $\Sigma$ is entirely determined by its genus $g$ and number of ends $n$, and $\chi(\Sigma) = 2-2g -n$.The space of marked complete hyperbolic structures on $\Sigma$, that is Teichm\"uller space, will be denoted $\T(\Sigma)$, and should be thought of as a continuous deformation space of hyperbolic metrics. For the purpose of simplicity, we ask the metrics to be geodesically complete, and so the ends of $X\in \T(\Sigma)$ are realized as cusps. The underlying moduli space, that is the space of hyperbolic structures on $\Sigma$ up to isometry, will be denoted $\M(\Sigma)$. $\Mod(\Sigma)$ will be the (full) mapping class group of $\Sigma$, that is the group of self-homeomorphisms of $\Sigma$ up to isotopy. This group acts on $\T(\Sigma)$ and its quotient is $\M(\Sigma)$. 

Formally a curve is the continuous image of a circle on $\Sigma$, but we will only be interested in a curve up to free homotopy. In particular, we only consider essential curves, that is those non-homotopic to a point or a boundary. A multicurve is a (finite) collection of curves. Associated to a multicurve $\Gamma$ is a function which associates to $X\in \T(\Sigma)$ the quantity $\ell_\Gamma(X)$, the length of the unique geodesic representative of $\Gamma$ on $X$. These length functions are continuous, analytic and in fact convex \cite{Kerckhoff, Wolpert}. 

Intersection between curves is defined as minimal intersection among representatives, and is denoted $i(\cdot,\cdot)$. A curve is simple if it has no self-intersections. A pants decomposition is a maximal collection of disjoint and distinct simple closed curves, and decomposes the surface into three-holed spheres (pairs of pants). The boundary curves of a pair of pants are sometimes called cuffs.

The following result \cite{Keen} is stated here in non-quantitative terms for simple closed geodesics. Note that a version also holds for non-simple closed geodesics as well \cite{BasmajianCollar}, with the notable difference being that you cannot pinch a non-simple closed curve. 

\begin{lemma}[Collar lemma]\label{lem:collar}
A simple closed geodesic of length $\ell$ on a hyperbolic surface $X$ admits a cylindrical neighborhood (its collar) of positive width $w(\ell)$ which only depends on $\ell$ and such that $w(\ell) \to \infty$ when $\ell \to 0$. 
\end{lemma}

Our main use of the above result will be to consider hyperbolic structures where curves of a pants decomposition have length tending towards $0$, and hence all curves that are not among the cuffs of the given pants decomposition have length that tend to infinity. 

Given $X\in \T(\Sigma)$, the length of its shortest essential curve is its systole and is denoted $\sys(X)$. Unless $X$ is a pair of pants, the systole is realized by a simple closed geodesic. For $\epsilon>0$, the $\epsilon$-thick part of Teichm\"uller space is the subset $\T^{\epsilon}(\Sigma)\subset \T(\Sigma)$ consisting of surfaces with $\sys(X)\geq \epsilon$. By Mahler's compactness criterion, the corresponding thick part of moduli space $\M^{\epsilon}(\Sigma)$ is compact.

In this paper, we study length inequalities, which here will be upper bounds for lengths of curves or multicurves with given properties. Generally these inequalities will be about a topological type of curve or multicurve: two multicurves $\Gamma$ and $\Gamma'$ are of the same type if there exists $\phi\in \Mod(\Sigma)$ such that $\phi(\Gamma)=\Gamma'$. Given a multicurve $\Gamma$, we can consider its mapping class group orbit. These orbits divide the space of all multicurves into equivalence classes
$$
[\Gamma]:= \{ \phi(\Gamma) \mid \phi \in \Mod(\Sigma)\}
$$
sorted by type. Now given $X\in \T(\Sigma)$, we can study the length of a minimal representative of an equivalence class:
$$
\ell_{[\Gamma]}(X) = \min_{\phi \in \Mod(\Sigma)} \big\{ \ell_X(\phi(\Gamma)\big\}
$$
The existence of a minimum follows from the discreteness of the length spectrum. This function, due to its mapping class group invariance, descends nicely to $\M(\Sigma)$. Note that although the function $\ell_{[\Gamma]}(\cdot)$ remains continuous over $\M(\Sigma)$, it is no longer smooth. This is due to possible changes of homotopy classes realizing the minimum length of in an equivalence class.

The following result is a non-explicit general bound that holds for any topological type of multicurve. The proof is a compactness argument. The constant $K$ in the statement depends on the topology of $\Sigma$, the topological type of $\Gamma$ and a lower bound on the systole.

\begin{proposition}\label{prop:conditionallength}
Let $[\Gamma]$ be a type of multicurve on $\Sigma$. For any $\epsilon>0$, there exists a constant $K=K(\epsilon, \Sigma, [\Gamma])$ such that for any $X\in \T^{\epsilon}(\Sigma)$, we have
$$
\ell_{[\Gamma]}(X) \leq K.
$$
\end{proposition}

\begin{proof}
The space $\M^{\epsilon}(\Sigma)$ is compact, hence the continuous function $\ell_{[\Gamma]}(\cdot)$ admits a maximum on $\M^{\epsilon}(\Sigma)$. This maximum value is exactly $K$. 
\end{proof}

As an explicit example of the above result, consider the following result due to Buser and S\"eppala \cite{Buser-Seppala}. For a closed surface $\Sigma$ of genus $g\geq 2$, they consider {\it canonical homology bases}, that is collections of $2g$ simple closed curves $\{\alpha_i,\beta_i\}_{1\leq i \leq g}$ that satisfy:
\begin{enumerate}[a)]
\item
$i(\alpha_i,\beta_j)=\delta_{ij}$ for all $i,j\in \{1,\hdots,g\}$ (where $\delta_{ij}$ is the Kronecker delta),
\item
$i(\alpha_i,\alpha_j)=i(\beta_i,\beta_j)=0$ for $i\neq j$. 
\end{enumerate}
Note that such a system of curves automatically generate integer homology. See Figure \ref{fig:homology} for an illustration in genus $2$.

\begin{figure}[htbp]
\leavevmode \SetLabels
\L(.23*.5) $\alpha_1$\\%
\L(.32*.27) $\beta_1$\\%
\L(.74*.5) $\alpha_2$\\%
\L(.66*.27) $\beta_2$\\%
\endSetLabels
\begin{center}
\AffixLabels{\centerline{\includegraphics[width=8cm]{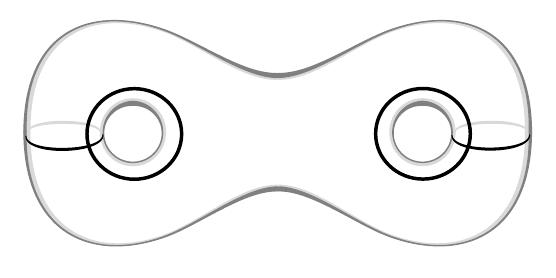}}}
\caption{A canonical homology basis in genus $2$}
\label{fig:homology}
\end{center}
\vspace{-0.8cm}
\end{figure}
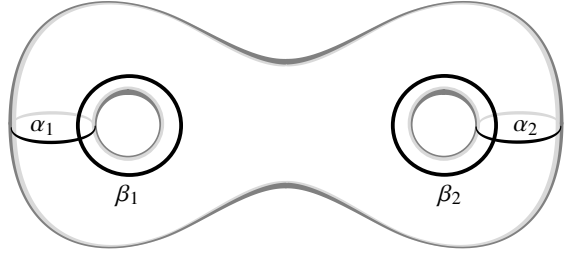

They prove that any $X\in \T^{\epsilon}(\Sigma)$ admits a canonical homology basis with all curves of length at most
$$
(g-1)\left(45+6\,\arcsinh\frac{1}{\epsilon}\right).
$$
This is an improvement over an earlier quantification in \cite{Buser-Seppala2}. The above bound shows that the $K$ in this instance is bounded above by 
$$
2g(g-1)\left(45+6\,\arcsinh\frac{1}{\epsilon}\right)
$$
because there are $2g$ curves in the family. Of course, for this to be an exact quantification of the constant $K$, it would have to be sharp (which it is not). Exact quantifications are rarely known but, as proved in \cite{Buser-Seppala}, the real constant must depend on $\epsilon$. Alternatively, this dependency on $\epsilon$ can also be deduced from Theorem \ref{thm:main}.

Naturally, this leads to the existence of upper bounds which only depend on topology and curve type, but not on systole length. A multicurve type $[\Gamma]$ is said to satisfy a strong length inequality if $\ell_{[\Gamma]}(\cdot)$ is upper bounded over $\M(\Sigma)$. By continuity of the length function $\ell_{[\Gamma]}(\cdot)$, this is equivalent to the existence of a surface $X_{\max}\in \M(\Sigma)$ such that 
$$
\max_{X\in \M(\Sigma)} \ell_{[\Gamma]}(X) = \ell_{[\Gamma]}(X_{\max}).
$$
To see this, one must show that the supremum of the length function cannot be reached on the boundary of moduli space, that is on a noded surface. As it turns out, as a consequence of Lemma \ref{lem:stretchpants} from the next section, a (finite) supremum is always reached in the ''thick" part of moduli space, and so in particular the $\sup$ is indeed a $\max$.

\section{Stretching pants, Bers' theorem and consequences}

The following lemma is by now well-known, but a sketch proof is provided for completeness. The proof uses strip maps, introduced by Thurston \cite{Thurston}, and used to great effect by many authors \cite{Danciger-Gueritaud-Kassel, Gueritaud, Papadopoulos-Theret, Parlier} to study of deformations of hyperbolic structures. Recall that a hyperbolic pair of pants is uniquely determined by its cuff lengths. In order to allow pants with cusp boundary, we use the convention that a cusp is a cuff of $0$ length. 

\begin{lemma}[Pants stretching lemma]\label{lem:stretchpants}
Let $Y_{x,y,z}$ be the unique hyperbolic pair of pants with cuff lengths $x,y,z \geq 0$. Then, for any (non-boundary) homotopy class of closed curve $\gamma$ on a pair of pants, and any $\epsilon>0$, we have
$$
\ell_\gamma\left(Y_{x,y,z}\right) < \ell_\gamma\left(Y_{x+\epsilon,y,z}\right)
$$
\end{lemma}
\begin{proof}[Sketch proof]
Consider on $Y=Y_{x+\epsilon,y,z}$, the simple orthogeodesic $a$ with both endpoints on the boundary curve of length $x+\epsilon$ (see Figure \ref{fig:arc}). 

\begin{figure}[htbp]
\leavevmode \SetLabels
\L(.445*.5) $a$\\%
\endSetLabels
\begin{center}
\AffixLabels{\centerline{\includegraphics[width=5cm]{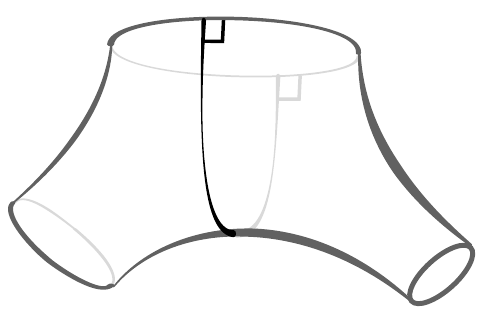}}}
\caption{The orthogeodesic $a$}
\label{fig:arc}
\end{center}
\vspace{-0.8cm}
\end{figure}

Note the closed geodesic $\gamma$ intersects $a$ at least once. Exactly like in the collar lemma, $a$ admits an embedded neighborhood, topologically a strip (see Figure \ref{fig:strip}). The idea is to now remove at least part of this strip to reduce the length of the boundary curve. 

\begin{figure}[htbp]
\leavevmode \SetLabels
\L(.445*.5) $a$\\%
\endSetLabels
\begin{center}
\AffixLabels{\centerline{\includegraphics[width=5cm]{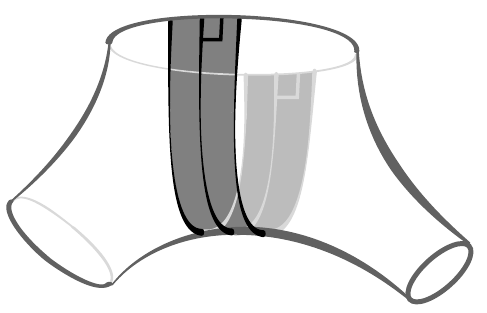}}}
\caption{A strip surrounding $a$}
\label{fig:strip}
\end{center}
\vspace{-0.8cm}
\end{figure}

To do this properly, it is more convenient to consider the complete pair of pants by adding funnels. The arc $a$ can now be extended into a complete simple and infinite length geodesics. In addition, because the boundary curve of length $x+\epsilon$ is not $0$, there is a family of simple complete geodesics parallel to each other and to $a$. We can take any two of these, cut away the strip between them, and paste them together to obtain a (complete) hyperbolic metric, say $Y'$ (see Figure \ref{fig:operation}). (The slightly sketchy part is here: in fact it is possible by a variational argument to show that this can be done such that the length of the new boundary component of $Y'$ is exactly $x$, but we won't dwell on this, the main point being that the boundary length has been reduced.)

\begin{figure}[H]
\leavevmode \SetLabels
\endSetLabels
\begin{center}
\AffixLabels{\centerline{\includegraphics[width=10cm]{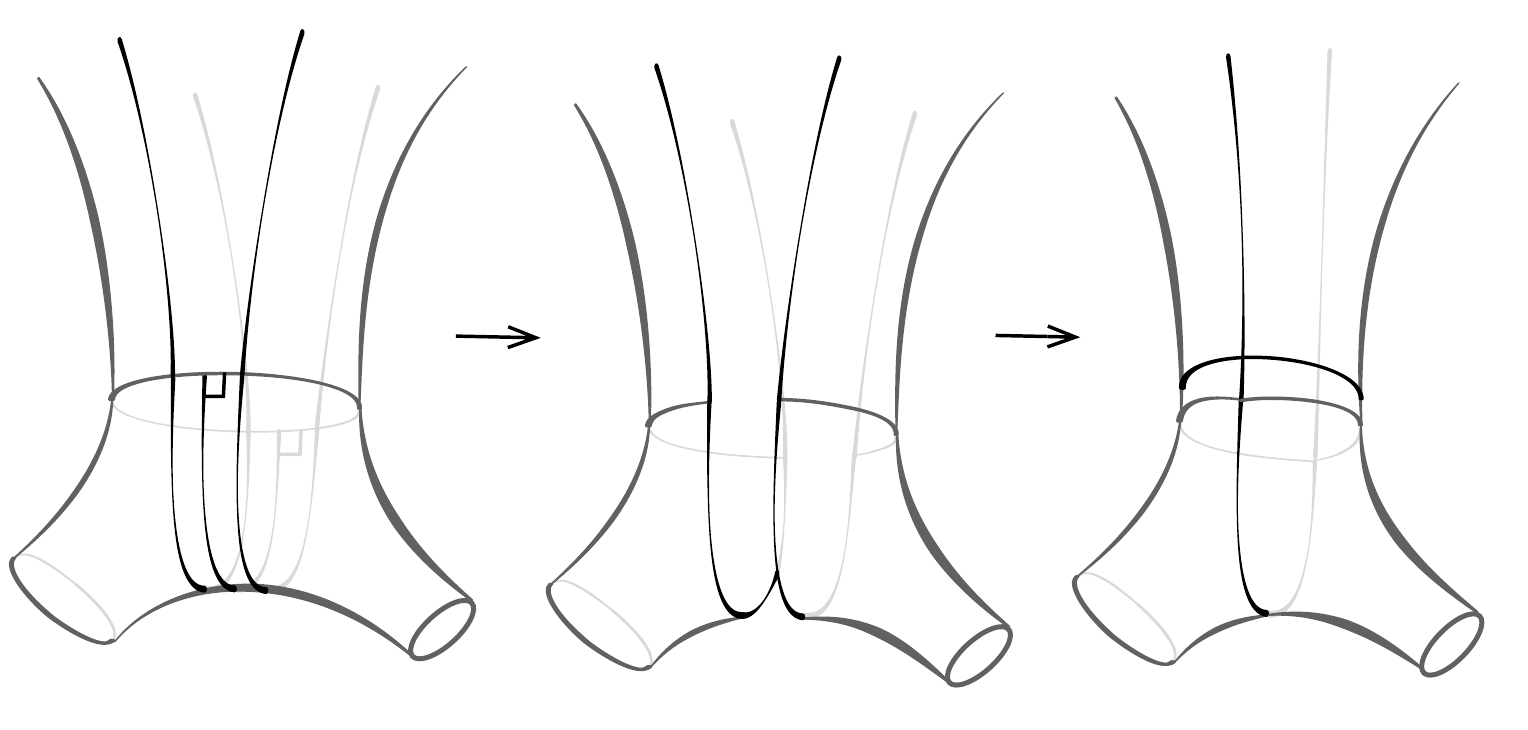}}}
\caption{Removing the strip neighborhood $\mathcal{S}(a)$}
\label{fig:operation}
\end{center}
\vspace{-0.8cm}
\end{figure}

We denote the strip enclosed by these geodesics by $\mathcal{S}(a)$.

Now can analyse the result of this operation on the length of $\gamma$. Note that, due to the topology of the strip neighborhood $\mathcal{S}(a)$ of $a$, $\gamma \cap \mathcal{S}(a)$ is a union of simple geodesic arcs. We look at each one. In order to find a representative of $\gamma$ on $Y'$, we replace each simple geodesic arc with its projection to $a$ (see Figure \ref{fig:project}). The point is that the projection strictly reduces lengths.
\begin{figure}[H]
\leavevmode \SetLabels
\endSetLabels
\begin{center}
\AffixLabels{\centerline{\includegraphics[width=5cm]{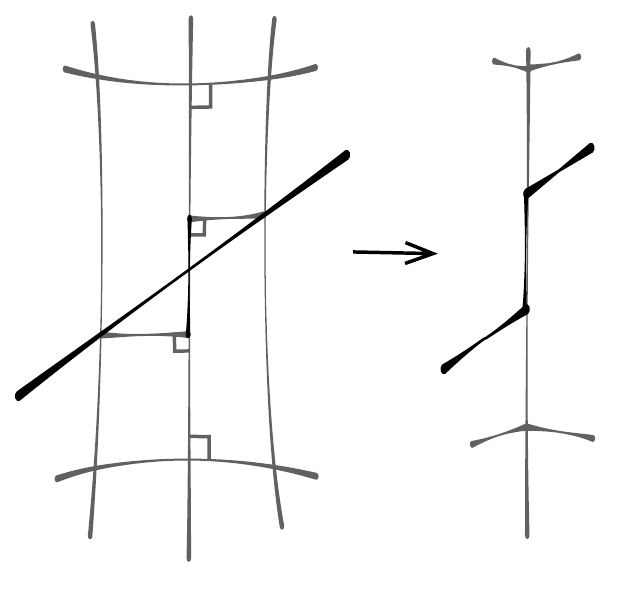}}}
\caption{A local picture of $\gamma$ under the operation}
\label{fig:project}
\end{center}
\vspace{-0.8cm}
\end{figure}
Thus, this results in a curve on $Y'$, in the same homotopy class, and of length strictly smaller. The corresponding geodesic is thus of length strictly smaller than it was previously. Hence we have $\ell_\gamma\left(Y\right) > \ell_\gamma\left(Y'\right)$ as required. 
\end{proof}
To make the proof fully rigorous 
We will mainly need an immediate corollary of the above result.
\begin{corollary}\label{cor:longpants}
For $L>0$, let $Y$ be a pair of pants with cuff lengths between $0$ and $L$. Let $Y_L$ be the pair of pants with all cuff lengths exactly $L$. Then, for any interior homotopy class of curve $\gamma$ on the pair of pants, we have $\ell_\gamma(Y_L) \geq \ell_\gamma(Y)$ with equality only occurring if $Y=Y_L$. 
\end{corollary}

Pants decompositions play an essential role in this story. The following result, originally due to Bers \cite{Bers1,Bers2}, has been since quantified by different authors \cite{BuserBook,BPS,ParlierShort}.

\begin{theorem}[Bers' constants]\label{thm:bers}
There exists a constant $B(\Sigma)$ such that any $X\in \T(\Sigma)$ admits a pants decomposition with all curves of length at most $B(\Sigma)$.
\end{theorem}

Note that this result does not fall in the framework of either Proposition \ref{prop:conditionallength} or Theorem \ref{thm:main}. Indeed, in order to find a short pants decomposition, we are allowed to choose any topological type of pants decomposition. The number of these grows with topology (for instance there are roughly $\sim g^g$ different types when $\Sigma$ is of genus $g$ and $g$ is large). So here we are not only minimizing among mapping class group orbits of a fixed multicurve, but we are minimizing among multicurves that belong to a family. The following result holds for all such families of multicurves. A family will be denoted $\{\Gamma_\alpha\}_{\alpha\in I}$ where $I$ is an index set (possibly infinite, but countable as there are only countably many topological types of finite multicurve on a finite type surface). 

\begin{theorem}\label{thm:maingeneral}
Let $\{\Gamma_\alpha\}_{\alpha\in I}$ be a family of multicurves. Then the quantity 
$$\max_{X\in \M(\Sigma)} \min_{\alpha\in I} \ell_{[\Gamma_\alpha]}(X)$$
is finite if and only if, for every pants decomposition $P$, there exists $\alpha\in I$ and $\phi \in \Mod(\Sigma)$ such that $i(P, \phi(\Gamma_\alpha)) = 0$. 
\end{theorem}

The statement might seem a little confusing at first, due to the fact that we are taking a maximum among hyperbolic structures of a double minimum (over a family and then over the mapping class group orbit). If the family is reduced to a single curve, the statement becomes Theorem \ref{thm:main} from the introduction. The proofs of both statements are identical, so we prove the more general statement above.

\begin{proof}[Proof of Theorem \ref{thm:maingeneral}]

We begin with the more straightforward direction, showing that if there is a pants decomposition $P$ which is intersected by any mapping class group orbit of any multicurve in the family $\{\Gamma_\alpha\}_{\alpha\in I}$, then there is no upper bound on the length. This follows directly from Lemma \ref{lem:collar} (the collar lemma). Indeed, by considering a sequence of hyperbolic structures with all curves in the pants decomposition of $P$ that converge to $0$, the length of any curve that intersects one of the pants curve necessarily goes to $+\infty$. As, by hypothesis, there is at least one curve in every multicurve $\phi(\Gamma_\alpha)$, for all $\phi \in \Mod(\Sigma)$ and all $\alpha \in I$, that intersects a curve in $P$, the result follows. 

The other direction will follow from Bers' theorem (Theorem \ref{thm:bers}) and Corollary \ref{cor:longpants} above.

Take any $X\in \T(\Sigma)$. By Bers' theorem, there exists a pants decomposition of $X$, say $P$, with all curves of length at most $B(\Sigma)$. Now, by hypothesis, there exists $\Gamma_\alpha$ such that $i(\Gamma_\alpha,P)= 0$. Note that $\Gamma_\alpha$ can contain curves from $P$, but none that intersect curves of $P$ transversally. 

By hypothesis all curves that belong to both $\Gamma_\alpha$ and $P$ are upper bounded by $B(\Sigma)$. We now the corollary to all remaining curves of $\Gamma_\alpha$. Indeed, any such a curve $\gamma \in \Gamma_\alpha$ is contained in a pair of pants. By hypothesis, the cuff lengths of this pair of pants are at most $B(\Sigma)$. By Corollary \ref{cor:longpants}, $\ell_X(\gamma)$ is at most the length when the pair of pants has all cuff lengths equal to $B(\Sigma)$, which is some finite number that depends on $B(\Sigma)$ and the topological type of $\gamma$. As there are a finite number of such curves, the result follows.
\end{proof}

There are many instances of the above theorem, each obtained by changing the multicurve or family of multicurves. Note that only the topological type of a multicurve matters in the statement. Some of these results are in the positive direction: the theorem implies that there is an upper bound on a minimal length representative that only depends on the topology of $\Sigma$. Others are in the negative direction, that is that certain multicurves, or families of multicurves, do not admit such upper bounds. We give a (non-exhaustive) list of results of this type.

\begin{enumerate}[I.]
\item Even though it was used in the proof, Bers' theorem is an example obtained by taking the family of multicurves to be the full set of pants decompositions (or simply one pants decomposition for each topological type). To put in the framework of the theorem, the length of a pants decomposition should be defined as the sum, and not the maximum value. 

The quantification of the implied constants - for both the sum and the max - has attracted some attention over the years, but it does not seem to be an easy problem. In fact, even the rough growth in terms of genus is not known \cite{BuserBook,ParlierShort,GPY}. On the positive side, the exact constant in genus $2$ is known by a result of Gendulphe \cite{Gendulphe}, as is the rough growth in terms of the number of cusps \cite{Balacheff-Parlier,BPS}. 

\item One can also take the full set of all topological types of all multicurves: this boils down to the systolic inequality. Quantifying the exact constants that appear is an arduous task. For orientable closed surfaces, the only constant known is again in genus $2$ \cite{Jenni}. Buser and Sarnak \cite{Buser-Sarnak} showed that the constants must grow logarithmically in genus, and since then there have been multiple variations and refinements of this (see for instance \cite{BFR,KSV,Fanoni-Parlier, SchmutzSchaller}), including generalizations in the world of variable curvature and higher dimensional manifolds \cite{Gromov}.

\item A slight variation of the above is to look at the homological systole of a closed surface $\Sigma$. That is, the shortest curve that is not only non-trivial in homotopy, but also in homology. By cut and paste arguments, the homological systole is always realized by a non-separating simple closed curve. Hence, it is always in the same mapping class orbit. As above, few optimal constants are known. However, for closed surfaces, it is known that the optimal constants are equal to those from the systolic inequality. \cite{ParlierPapa} Said differently, the systole of a maximal surface is homologically non-trivial.

\item In addition to the results of Buser and Sepp\"al\"a mentioned previously, one can try and bound families of homologically independent curves, but that do not necessarily form (part of) a canonical basis. If one requests a full homology basis, by the theorem above, there is no upper bound on its length over moduli space. However, by Bers' theorem, and the observation that any pants decomposition contains $g$ homologically independent curves, one can find an upper bound on the length of up until $g$ curves by a function of topology. So what about $g+1$ curves?

Gromov observed \cite[Section 5]{GromovFilling} that any minimal length homology basis consists of simple curves that pairwise intersect at most once. Now given $g+1$ homologically independent and simple curves, there must be at least a pair that intersect. And a pair of intersecting simple curves necessarily intersects {\it all} pants decompositions. Hence by the theorem above, there is no upper bound for such a family and so strong length inequality stops at exactly $g$ homologically independent curves. 

More precise quantifications of the constants have also attracted attention. In particular the Buser-Sarnak logarithmic bound can be extended to roughly $ag$ curves for any $a<1$ \cite{BPS}. 

Finally note that Gromov's observation above (on the intersection properties of minimal bases) still forces one to consider multiple, although finite, topological types of multicurves.

\item In a somewhat opposite direction, consider $\{\Gamma\}_\alpha$ to be the set of separating simple closed curves on a closed surface $\Sigma$ of genus $g$. Then, as there exists a pants decomposition consisting only of non-separating curves, it will essentially intersect any element of the mapping class group orbit of any element of $\Gamma$. Hence, for all $\alpha\in I$, the function $\ell_{[\Gamma_\alpha]}(\cdot)$ admits no upper bound over $\M(\Sigma)$. If however one takes the larger set $\{\Gamma\}_\alpha$ of all homologically trivial curves (but not homotopically trivial), then it is a consequence of a theorem of Sabourau \cite{Sabourau} that this admits an upper bound that only depends on genus. This subtle difference lead to a small gap in Mirzakhani's study of the expected value of the shortest {\it simple} homologically trivial curve in \cite{Mirzakhani}, but this was not the difficult part of her work and has since been cleared up \cite{PWX}.

\item For multicurves containing more than one curve, many of these problems admit variations. Here we suppose the length of a multicurve is the sum of lengths of its components, but of course one could also consider other variations, such as the $\max$ or the product. For instance, the product of lengths of homologically distinct curves was studied in \cite{Balacheff-Karam-Parlier}. Replacing the sum of the lengths with the maximum will also satisfy the same boundedness condition of Theorem \ref{thm:maingeneral}, but products - or other functions of the lengths of the components - need to be checked on a case-by-case basis. 
\end{enumerate}

\ack{This work was supported by the Luxembourg National Research Fund OPEN grant O19/13865598.}



\begin{thebibliography}{99}

\bibitem{Balacheff-Karam-Parlier} Balacheff, Florent, Karam, Steve and Parlier, Hugo. The minimal length product over homology bases of manifolds. Math. Ann. 380 (2021), no. 1-2, 825--854.

\bibitem{Balacheff-Parlier} Balacheff, Florent and Parlier, Hugo. Bers' constants for punctured spheres and hyperelliptic surfaces. J. Topol. Anal. 4 (2012), no. 3, 271--296.

\bibitem{BPS} Balacheff, Florent, Parlier, Hugo and Sabourau, St\'ephane. Short loop decompositions of surfaces and the geometry of Jacobians. Geom. Funct. Anal. 22 (2012), no. 1, 37--73.

\bibitem{BasmajianCollar} Basmajian, Ara. The stable neighborhood theorem and lengths of closed geodesics. Proc. Amer. Math. Soc. 119 (1993), no. 1, 217--224.

\bibitem{BasmajianUniversal} Basmajian, Ara. Universal length bounds for non-simple closed geodesics on hyperbolic surfaces. J. Topol. 6 (2013), no. 2, 513--524.

\bibitem{Bers1} Bers, Lipman. Spaces of degenerating Riemann surfaces. Discontinuous groups and Riemann surfaces (Proc. Conf., Univ. Maryland, College Park, Md., 1973), pp. 43--55. Ann. of Math. Studies, No. 79, Princeton Univ. Press, Princeton, N.J., 1974.

\bibitem{Bers2} Bers, Lipman. An inequality for Riemann surfaces. Differential geometry and complex analysis, 87--93, Springer, Berlin, 1985.

\bibitem{BFR} Bourque, Maxime Fortier and Rafi, Kasra. Local maxima of the systole function. J. Eur. Math. Soc. (JEMS), to appear. 

\bibitem{Brock} Brock, Jeffrey F. The Weil-Petersson metric and volumes of 3-dimensional hyperbolic convex cores. J. Amer. Math. Soc. 16 (2003), no. 3, 495--535.

\bibitem{BuserBook} Buser, Peter. Geometry and spectra of compact Riemann surfaces. Progress in Mathematics, 106. Birkh\"auser Boston, Inc., Boston, MA, 1992.

\bibitem{Buser-Sarnak} Buser, Peter and Sarnak, Peter. On the period matrix of a Riemann surface of large genus. Invent. Math. 117 (1994), no. 1, 27--56.

\bibitem{Buser-Seppala} Buser, Peter and Sepp\"al\"a, Mika. Triangulations and homology of Riemann surfaces. Proc. Amer. Math. Soc. 131 (2003), no. 2, 425--432.

\bibitem{Buser-Seppala2} Buser, Peter and Sepp\"al\"a, Mika. Short homology bases and partitions of Riemann surfaces. Topology 41 (2002), no. 5, 863--871.

\bibitem{Cavendish-Parlier} Cavendish, William and Parlier, Hugo. Growth of the Weil-Petersson diameter of moduli space. Duke Math. J. 161 (2012), no. 1, 139--171.

\bibitem{Danciger-Gueritaud-Kassel} Danciger, Jeffrey, Gu\'eritaud, Fran\c{c}ois, and Kassel, Fanny. Margulis spacetimes via the arc complex. Invent. Math. 204 (2016), no. 1, 133--193.

\bibitem{Erlandsson-Souto} Erlandsson, Viveka and Souto, Juan. Mirzakhani's Curve Counting and geodesic currents. Progress in Mathematics, Springer Birkh\"auser, to appear. 

\bibitem{Fanoni-Parlier} Fanoni, Federica and Parlier, Hugo. Systoles and kissing numbers of finite area hyperbolic surfaces. Algebr. Geom. Topol. 15 (2015), no. 6, 3409--3433.

\bibitem{Gendulphe} Gendulphe, Matthieu. Constante de Bers en genre 2. (French) [Bers constant of genus 2] Math. Ann. 350 (2011), no. 4, 919--951.

\bibitem{Gueritaud} Gu\'eritaud, Fran\c{c}ois. Strip maps of small surfaces are convex. Illinois J. Math. 60 (2016), no. 1, 19--37.

\bibitem{GPY} Guth, Larry, Parlier, Hugo and Young, Robert. Pants decompositions of random surfaces. Geom. Funct. Anal. 21 (2011), no. 5, 1069--1090.

\bibitem{GromovFilling} Gromov, Mikhael. Filling Riemannian manifolds. J. Differential Geom. 18 (1983), no. 1, 1--147. 

\bibitem{Gromov} Gromov, Misha. Metric structures for Riemannian and non-Riemannian spaces. Based on the 1981 French original. With appendices by M. Katz, P. Pansu and S. Semmes. Translated from the French by Sean Michael Bates. Reprint of the 2001 English edition. Modern Birkh\"auser Classics. Birkh\"auser Boston, Inc., Boston, MA, 2007.

\bibitem{Jenni} Jenni, Felix.
\"Uber den ersten Eigenwert des Laplace-Operators auf ausgew\"ahlten Beispielen kompakter Riemannscher Fl\"achen [On the first eigenvalue of the Laplace operator on selected examples of compact Riemann surfaces]. Comment. Math. Helv. 59 (1984), no. 2, 193--203. 

\bibitem{KSV} Katz, Mikhail G., Schaps, Mary and Vishne, Uzi. Logarithmic growth of systole of arithmetic Riemann surfaces along congruence subgroups. J. Differential Geom. 76 (2007), no. 3, 399--422.

\bibitem{Keen} Keen, Linda. Collars on Riemann surfaces. Discontinuous groups and Riemann surfaces (Proc. Conf., Univ. Maryland, College Park, Md., 1973), pp. 263--268. Ann. of Math. Studies, No. 79, Princeton Univ. Press, Princeton, N.J., 1974.

\bibitem{Kerckhoff} Kerckhoff, Steven P. The Nielsen realization problem. Bull. Amer. Math. Soc. (N.S.) 2 (1980), no. 3, 452--454.

\bibitem{Mirzakhani} Mirzakhani, Maryam. Growth of Weil-Petersson volumes and random hyperbolic surfaces of large genus. J. Differential Geom. 94 (2013), no. 2, 267--300.

\bibitem{Papadopoulos-Theret} Papadopoulos, Athanase and Th\'eret, Guillaume. Shortening all the simple closed geodesics on surfaces with boundary. Proc. Amer. Math. Soc. 138 (2010), no. 5, 1775--1784.

\bibitem{Parlier} Parlier, Hugo. Lengths of geodesics on Riemann surfaces with boundary. Ann. Acad. Sci. Fenn. Math. 30 (2005), no. 2, 227--236. 

\bibitem{ParlierPapa} Parlier, Hugo. The homology systole of hyperbolic Riemann surfaces. Geom. Dedicata 157 (2012), 331--338.

\bibitem{ParlierShort} Parlier, Hugo. A short note on short pants. Canad. Math. Bull. 57 (2014), no. 4, 870--876.

\bibitem{PWX} Parlier, Hugo, Wu, Yunhui and Xue, Yuhao. The simple separating systole for hyperbolic surfaces of large genus. J. Inst. Math. Jussieu, to appear.

\bibitem{Rafi} Rafi, Kasra. A combinatorial model for the Teichmüller metric. Geom. Funct. Anal. 17 (2007), no. 3, 936--959.

\bibitem{Rafi-Tao} Rafi, Kasra and Tao, Jing. The diameter of the thick part of moduli space and simultaneous Whitehead moves. Duke Math. J. 162 (2013), no. 10, 1833--1876.

\bibitem{Sabourau} Sabourau, St\'ephane. Asymptotic bounds for separating systoles on surfaces. Comment. Math. Helv. 83 (2008), no. 1, 35--54.

\bibitem{SchmutzSchaller} Schmutz Schaller, Paul. Geometry of Riemann surfaces based on closed geodesics. Bull. Amer. Math. Soc. (N.S.) 35 (1998), no. 3, 193--214.

\bibitem{Thurston} Thurston, William. Minimal stretch maps between hyperbolic surfaces, preprint, 1986.

\bibitem{Wolpert} Wolpert, Scott A. Geodesic length functions and the Nielsen problem. J. Differential Geom. 25 (1987), no. 2, 275--296.

\end{thebibliography}
\end{document}